\documentclass[12pt,a4paper]{amsart}
\usepackage{times}
\usepackage{amssymb}
\usepackage{mathtools}
\usepackage{amsmath}
\usepackage{amsthm}
\usepackage{rotating}
\usepackage{enumitem}
\usepackage[all]{xy}
\usepackage{tikz}
\usetikzlibrary{calc}
\usepgflibrary{shapes.geometric}
\usepgflibrary{shapes.misc}
\usetikzlibrary{positioning}
\usetikzlibrary{decorations}
\usetikzlibrary{decorations.pathreplacing}
\usetikzlibrary{shapes,snakes}
\usetikzlibrary{arrows}
\usepackage{url}
\usepackage[T1]{fontenc}
\usepackage[english]{babel}
\usepackage[top=1in,bottom=1in,right=1in,left=1in]{geometry}

\usepackage[colorinlistoftodos]{todonotes}
\usepackage{xcolor}

\definecolor{sira}{HTML}{ED9090}

\usepackage{hyperref}

\newtheorem{lemma}{Lemma}[section]
\newtheorem{proposition}[lemma]{Proposition}
\newtheorem{corollary}[lemma]{Corollary}
\newtheorem{theorem}[lemma]{Theorem}
\newtheorem*{theorem*}{Theorem}

\theoremstyle{definition}
\newtheorem{definition}[lemma]{Definition}
\newtheorem{example}[lemma]{Example}

\newtheorem*{notation*}{Notation}

\theoremstyle{remark}
\newtheorem{remark}[lemma]{Remark}

\newcommand{\supp}{\operatorname{Supp}\nolimits}

\newcommand{\im}{\operatorname{im}\nolimits}

\newcommand{\id}{\operatorname{id}\nolimits}

\newcommand{\var}{\mathbf{var}}
\newcommand{\ex}{\mathbf{ex}}
\newcommand{\inv}{\mathbf{inv}}

\newcommand{\ZZ}{\mathbb{Z}}

\newcommand{\J}{\mathcal{J}}

\newcommand{\CC}{\mathcal{C}}

\newcommand{\integ}{\mathbb{Z}}
\newcommand{\cross}{\times}
\newcommand{\curly}[1]{\mathcal{#1}}
\newcommand{\lgen}{\ensuremath \mathopen{<}} 
\newcommand{\rgen}{\ensuremath \mathclose{>}}
\newcommand{\union}{\cup}

\newcommand{\cl}{\operatorname{cl}\nolimits}
\newcommand{\Clus}{\operatorname{\textbf{\textup{Clus}}}\nolimits}
\newcommand{\disjointunion}{\sqcup}
\newcommand{\bigdisjointunion}{\bigsqcup}
\newcommand{\bigtensor}{\bigotimes}
\newcommand{\bigdsum}{\bigoplus}
\newcommand{\iso}{\cong}
\newcommand{\Gr}{\operatorname{Gr}\nolimits}

\title{Graded quantum cluster algebras of infinite rank as colimits} 
\author{Jan E. Grabowski}
\address{Department of Mathematics and Statistics, Lancaster University, Lancaster, LA1 4YF, United Kingdom}
\email{j.grabowski@lancaster.ac.uk}
\urladdr{http://www.maths.lancs.ac.uk/$\sim$grabowsj/}
\author{Sira Gratz}
\address{Institut f\"{u}r Algebra, Zahlentheorie und Diskrete Mathematik, Leibniz Universit\"{a}t Hannover, Welfengarten 1, 30167 Hannover, Germany}
\email{gratz@math.uni-hannover.de}
\date{14th October 2015}

\begin{document}

\begin{abstract}
We provide a graded and quantum version of the category of rooted cluster algebras introduced by Assem, Dupont and Schiffler and show that every graded quantum cluster algebra of infinite rank can be written as a colimit of graded quantum cluster algebras of finite rank.

As an application, for each $k$ we construct a graded quantum infinite Grassmannian admitting a cluster algebra structure, extending an earlier construction of the authors for $k=2$.
\end{abstract}

\maketitle

\section{Introduction}

Cluster algebras, as introduced by Fomin and Zelevinsky (\cite{FZ-CA1}), are certain commutative rings with a combinatorial structure on a distinguished set of generators, which can be grouped into overlapping sets of a given cardinality, called clusters. The original definition requires clusters to be finite but lately, for a variety of reasons, it has become important to allow infinite clusters, that is, cluster algebras of infinite rank.

Much of the motivation for considering infinite rank cluster algebras comes from representation theory: various types of categorifications of cluster algebras have been studied and these categories often naturally have structures corresponding to infinite clusters.  Examples include the additive categorifications due to Igusa and Todorov (\cite{IT:continuous},\cite{IT:cyclic}) and the monoidal categorifications of Hernandez and Leclerc (\cite{HL}).  In earlier work, we considered a category whose infinite rank cluster structure has been studied by Holm and J\o rgensen (\cite{HJ}) and showed that their category is related to a family of cluster algebras of infinite rank that are themselves related to the infinite Grassmannians of (2-)planes in a space of countably infinite dimension (\cite{Grabowski-Gratz}).  Other related work includes \cite{Gorsky} and \cite{Ndoune}.

Cluster algebras have also been generalised in other ways.  The notion of a graded cluster algebra has been used by various authors, more or less explicitly; a systematic exposition and exploration of some general properties may be found in \cite{GradedClusterAlgebras} by the first author.  We note that graded cluster algebras are particularly important when studying cluster algebra structures on the (multi-)homogeneous coordinate rings of projective varieties, as one naturally expects.

Another important generalisation concerns the quantization of cluster algebras, that is, a noncommutative version of the theory.  This was done by Berenstein and Zelevinsky (\cite{BZ-QCA}) and has been followed by work of a large number of authors demonstrating that many families of quantizations of coordinate rings of varieties have quantum cluster algebra structures.

In this work, our main aim is to bring these three parts of cluster algebra theory together, to study graded quantum cluster algebras of infinite rank.

More specifically, Assem, Dupont and Schiffler (\cite{ADS}) have introduced a category $\Clus$ of rooted cluster algebras. The objects of $\Clus$ are pairs consisting of a cluster algebra and a fixed initial seed. Fixing a distinguished initial seed allows for the definition of natural maps between cluster algebras, so-called rooted cluster morphisms, which are ring homomorphisms commuting with mutation and which provide the morphisms for the category $\Clus$.

In the first part of this paper, we will introduce a corresponding category for the graded quantum case.  We can loosely characterise the difference from the classical case as follows: there are more objects (since there are many choices of grading and/or quantization) but fewer morphisms (since the morphisms are required to preserve the extra structure), but the properties of the category are largely unchanged.

The main technical issues concern the definition of a rooted cluster morphism.  By following the lead given by the classical case, we obtain a natural definition but the intrinsic rigidity of the noncommutative setting means that it is in fact very constraining.  We ask that morphisms are algebra homomorphisms and so the quasi-commutation relations between variables in the same cluster must be preserved exactly.

In the second part, we generalise results of the second author (\cite{Gratz-Colimits}) on colimits.  In the classical case, one may show that the category $\Clus$ is neither complete nor cocomplete, that is, limits and colimits do not in general exist. However, it has sufficient colimits to express any cluster algebra of infinite rank as a colimit of cluster algebras of finite rank.  Here, we prove the latter result for graded quantum cluster algebras, our main theorem:

\begin{theorem}
Every graded quantum rooted cluster algebra of infinite rank can be written as a colimit of graded quantum rooted cluster algebras of finite rank.
\end{theorem}

As an application, we extend our previous construction of the infinite quantum Grassmannian $\curly{O}_{q}(\Gr(2,\infty))$ from \cite{Grabowski-Gratz}, where we relied heavily on properties of the infinite rank cluster category studied by Holm and J\o rgensen (\cite{HJ}).  Here we construct infinite graded quantum Grassmannians $\curly{O}_{q}(\Gr(k,\infty))$ for all $k$, as colimits of the quantum cluster algebra structures on Grassmannians $\curly{O}_{q}(\Gr(k,n))$ (\cite{GradedQCAs-Grkn}).

\subsection*{Acknowledgements}

The first author would like to thank the Institut f\"{u}r Algebra, Zahlentheorie und Diskrete Mathematik, Leibniz Universit\"{a}t Hannover for their hospitality in June 2014, during which time this work was begun. Part of this work has been carried out in the framework of the research priority programme SPP 1388 Darstellungstheorie of the Deutsche Forschungsgemeinschaft (DFG). The second author gratefully acknowledges financial support through the grant HO 1880/5-1

\section{Rooted cluster algebras} \label{S:rooted cluster algebras}

Cluster algebras were introduced by Fomin and Zelevinsky in \cite{FZ-CA1} and their quantum analogues by Berenstein and Zelevinsky in \cite{BZ-QCA}. Throughout this paper we work with (graded) (quantum) cluster algebras of geometric type.

The definition of a quantum cluster algebra requires some technical preparation and, in the interest of brevity, we refer the reader to the recent work of Goodearl and Yakimov (\cite{GoodearlYakimovQCA}) for a detailed exposition.  In particular we follow \cite{GoodearlYakimovQCA} in considering quantum cluster algebras with multiple quantum parameters, which generalises \cite{BZ-QCA}, where only the single parameter case was treated.  The main deviations from \cite{GoodearlYakimovQCA} will be that we will permit quantum seeds of arbitrary cardinality and that we will incorporate gradings into the picture from the start.

Let $M$ be a matrix whose rows are indexed by a set $R$ and columns by a set $C$.  Then given $I\subseteq R$ and $J \subseteq C$, we denote by $M|_{I}^{J}$ the submatrix of $M$ with rows indexed by $I$ and columns indexed by $J$.  If $I=R$, we write simply $M|^{J}$ and correspondingly write simply $M|_{I}$ if $J=C$.  Columns (respectively, rows) of a matrix $M$ will be denoted $M^{j}$ (resp.\ $M_{i}$).

\subsection{Seeds}\label{S:Seeds from triangulations of the closed disc}

Throughout, given $\mathbb{K}$-algebra elements $X_{1},\dotsc ,X_{r}$ and $a=(a_{1},\dotsc ,a_{r})^{T}\in \integ^{r}$ we set
\[ X^{a} = \prod_{i=1}^{r} X_{i}^{a_{i}} \]
as a shorthand for monomials, where the empty product is defined to be $1$.  If $a=(a_{i})_{i\in I}$ is an $I$-indexed family of integers with finite support (that is, the set $\supp(a)=\{ i \in I \mid a_{i}\neq 0\}$ is finite), we write $X^{a}=\prod_{i\in \supp(a)} X_{i}^{a_{i}}$.

Let $\var$ be a set; this set will be the indexing set for our cluster variables and all their associated data.  By mild abuse of notation, we denote by $\integ^{\var}$ the free Abelian group on a basis $\{ e_{i} \mid i\in \var \}$.  Note that, by the definition of an infinitely generated free Abelian group, $a\in \integ^{\var}$ may be identified with a $\var$-indexed family of integers with finite support.

A $(\var \cross \var)$-indexed family $\mathbf{q}=(q_{kj})_{k,j\in \var}$ of elements of $\mathbb{K}^{\ast}$ is called a multiplicatively skew-symmetric matrix\footnote{It is natural to regard $(I\cross I)$-indexed (or indeed, $(I \cross J)$-indexed) families of integers or elements of $\mathbb{K}$ as matrices, even when $I$ is not necessarily ordered, in line with the finite and countable situations, and we shall do so.} if $q_{kk}=1$ and $q_{jk}=q_{kj}^{-1}$ for all $j\neq k$.  There is an associated skew-symmetric bicharacter $\Omega_{\mathbf{q}}\colon \integ^{\var}\cross \integ^{\var} \to \mathbb{K}^{\ast}$ given by $\Omega_{\mathbf{q}}(e_{k},e_{j})=q_{kj}$.

A multiplicatively skew-symmetric matrix $\mathbf{q}$ as above gives rise to a quantum torus $\curly{T}_{\mathbf{q}}$.  This is defined to be the quotient of $\mathbb{K}\lgen Y_{i}^{\pm 1} \mid i\in \var \rgen$ (the noncommutative Laurent polynomial algebra generated by the $Y_{i}$) by the ideal generated by the set $\{ Y_{k}Y_{j}-q_{kj}Y_{j}Y_{k} \mid j,k\in \var \}$.  Then $\curly{T}_{\mathbf{q}}$ has a $\mathbb{K}$-basis $\{ Y^{a} \mid a\in \integ^{\var} \}$.

The subalgebra of $\curly{T}_{\mathbf{q}}$ generated by the set $\{ Y_{i} \mid i\in \var \}$ will be denoted $\curly{A}_{\mathbf{q}}$ and called quantum affine space.

Given a multiplicatively skew-symmetric matrix $\mathbf{r}=(r_{kj})$, we write $\mathbf{r}^{\cdot 2}=(r_{kj}^{2})$.  Then $\curly{T}_{\mathbf{r}^{\cdot 2}}$ has a $\mathbb{K}$-basis $\{ Y^{(a)}=\curly{S}_{\mathbf{r}}(a)Y^{a} \mid a\in \integ^{\var}\}$ where $\curly{S}_{\mathbf{r}}(a)=\prod_{j<k} r_{jk}^{-a_{j}a_{k}}$ (again, the empty product is defined to be 1).  It is straightforward to check that $Y^{(e_{i})}=Y_{i}$ for $i\in \var$ and $Y^{(a)}Y^{(b)}=\Omega_{\mathbf{r}}(a,b)Y^{(a+b)}$.  The pair consisting of $\curly{T}_{\mathbf{r}^{\cdot 2}}$ and the $\mathbb{K}$-basis $\{ Y^{(a)} \mid a\in \integ^{\var}\}$ will be called the \emph{based} quantum torus associated to $\mathbf{r}$.

This brings us to the definitions of a toric frame and an exchange matrix.

\begin{definition}[{\cite[Definition~2.2]{GoodearlYakimovQCA}}]\label{D:toric frame} Let $\curly{F}$ be a division algebra over $\mathbb{K}$.  A map $M\colon \integ^{\var}\to \curly{F}$ is a {\em toric frame} if there exists a multiplicatively skew-symmetric matrix $\mathbf{r}$ such that
\begin{enumerate}[label=\textup{(\alph*)}]
\item\label{d:toric-frame-a} there is an algebra embedding $\phi\colon \curly{T}_{\mathbf{r}^{\cdot 2}} \to \curly{F}$ given by $\phi(Y_{i})=M(e_{i})$ for all $i\in \var$,
\item $\phi(Y^{(a)})=M(a)$ for all $a\in \integ^{\var}$ and
\item $\curly{F}=\text{Fract}(\phi(\curly{T}_{\mathbf{r}^{\cdot 2}}))$.
\end{enumerate}
\end{definition}

\noindent Note that the toric frame $M$ determines the matrix $\mathbf{r}$ uniquely: for $j>i$ we set $\mathbf{r}_{ij} =\frac{M(e_i)M(e_j)}{M(e_i+e_j)}$ and write $\mathbf{r}(M)$ for this matrix.

\begin{definition} A $(\var \cross \var)$-indexed integer matrix $\tilde{B}$ is an {\em exchange matrix} if
\begin{enumerate}
\item $\tilde{B}$ is locally finite, that is, all rows and columns of $\tilde{B}$ have finite support when considered as $\var$-indexed families;
\item $\tilde{B}$ is sign-skew-symmetric: for all $i,j\in \var$, $\tilde{B}_{ij}\tilde{B}_{ji}\leq 0$ and $\tilde{B}_{ij}\tilde{B}_{ji}=0$ implies $\tilde{B}_{ij}=\tilde{B}_{ji}=0$.
\end{enumerate}
\end{definition}

Let us fix a subset $\ex \subseteq \var$ of \emph{exchangeable} indices; the elements of $\var\setminus\ex$ will be called frozen indices.  We are adopting the convention that exchange matrices $\tilde{B}$ have rows and columns indexed by $\var$; that is, in the quiver setup (where we associate a quiver to a skew-symmetric matrix in the usual way), we permit arrows between frozen vertices.  The $\ex \cross \ex$ submatrix of $\tilde{B}$ will be called the principal part and denoted $B$.

\begin{definition}[{\cite[\S 2.3]{GoodearlYakimovQCA}}]
Given a multiplicatively skew-symmetric matrix $\mathbf{r}$ and an exchange matrix $\tilde{B}$, we say that the pair $(\tilde{B},\mathbf{r})$ is {\em compatible} if the integers $t_{kj}=\Omega_{\mathbf{r}}(\tilde{B}^{k},e_{j})$ satisfy
\begin{align*}
& t_{kj}=1\ \text{for all}\ k\in \ex, j\in \var,k\neq j \\
& t_{kk}\ \text{are not roots of unity, for all}\ k\in \ex.
\end{align*}
\end{definition}

Compatible pairs may be mutated.  Recall (\cite{BFZ-CA3}) the matrices $E$ and $F$ associated to an exchange matrix $\tilde{B}$ and a choice of exchangeable index $k$ defined as follows:
\begin{align*}
E_{ij} & = \begin{cases} \delta_{ij} & \text{if}\ j\neq k \\ -1 & \text{if}\ i=j=k \\ \max(0,-b_{ik}) & \text{if}\ i\neq j=k \end{cases} \\
F_{ij} & = \begin{cases} \delta_{ij} & \text{if}\ i\neq k \\ -1 & \text{if}\ i=j=k \\ \max(0,b_{kj}) & \text{if}\ i=k\neq j \end{cases} 
\end{align*}
Since $\tilde{B}$ is locally finite, so are $E$ and $F$.

In \cite{BFZ-CA3} and \cite{GoodearlYakimovQCA}, signed versions of these matrices are defined and it is then shown that mutation does not depend on the choice of this sign, so we have given here only $E=E_{+}$ and $F=F_{+}$.  Also, in those papers, the more usual $\var \cross \ex$ exchange matrices are considered; however the extension of these to square matrices by the same formul\ae\ remains valid.

Then for a compatible pair $(\tilde{B},\mathbf{r})$ and $k\in \ex$, we may define the mutation of $\tilde{B}$ as the matrix $\mu_{k}(\tilde{B})=E\tilde{B}F$.  Here and throughout we use the usual notation for matrix multiplication: since at least one of the matrices concerned will always be locally finite, the standard definition $(AB)_{ij}=\sum_{k} a_{ik}b_{kj}$ remains valid.

We also define the mutation of $\mathbf{r}$ as 
\[ \mu_{k}(\mathbf{r})_{ij}=\prod_{k,l} b_{kl}^{E_{ki}E_{lj}} \]
It is shown in \cite[Proposition~2.6]{GoodearlYakimovQCA} that the pair $(\mu_{k}(\tilde{B}),\mu_{k}(\mathbf{r}))$ is again compatible, provided that the principal part of $\tilde{B}$ is skew-symmetrizable.  We will work in geometric type and so assume that the latter condition indeed holds.

Given an exchange matrix, we may define an associated grading matrix, as follows.

\begin{definition}\label{d:grading} Let $I$ be a set.  A {\em $\integ^{I}$-grading for an exchange matrix $\tilde{B}$} is a $(\var\cross I)$-indexed family of integers such that $\tilde{B}^{T}G=0$.  That is, for all $k\in \ex$, $\sum_{\substack{j\in \var\\ l\in I}} b_{jk}g_{jl}=0$.
\end{definition}

Gradings may also be mutated, by setting $\mu_{k}(G)=E^{T}G$ where $E$ is as defined previously.

\begin{remark}\label{R:classical} Although we introduce gradings at this early stage in the exposition, rather than as a later addition, the reader may---if they so wish---ignore references to the grading, by taking $G=0$ in the statements above and below.

Similarly, those interested in graded classical cluster algebras may recover that case by considering the multiplicatively skew-symmetric matrix with $\mathbf{r}(M)_{kj}=1$ for all $k\leq j$.
\end{remark}

The starting point for a graded quantum cluster algebra is a graded quantum seed, which consists of a toric frame, an exchange matrix satisfying a compatibility condition with the toric frame, a grading and (to facilitate certain definitions later) also records the sets $\var$ and $\ex$, as well as a third set $\inv$.

\begin{definition}[cf.\ \cite{BZ-QCA}] A {\em graded quantum seed} of a division algebra $\curly{F}$ is a tuple \hfill \hfill \linebreak $\Sigma=(M,\tilde{B},G,\var,\ex,\inv)$ where $M\colon \integ^{\var} \to \curly{F}$ is a toric frame, $\tilde{B}$ is an exchange matrix and $G$ is a grading for $\tilde{B}$ such that
\begin{enumerate}[label=\textup{(\alph*)}]
\item the pair $(\tilde{B},\mathbf{r}(M))$ is compatible, where $\mathbf{r}(M)$ is the matrix of $M$ and
\item the principal part of $\tilde{B}$, $\tilde{B}|_{\ex}^{\ex}$, is skew-symmetrizable.
\end{enumerate}
The sets $\var$ and $\ex$ are as above and $\inv$ is a fixed subset of $\var\setminus\ex$.
\end{definition}

The images of the standard basis vectors, $M(e_{i}) \in \curly{F}$, will be called the {\em quantum cluster variables of the quantum seed}.  The variables $M(e_{i})$ for $i\in \ex$ will be called {\em exchangeable} or {\em mutable}; the variables $M(e_{i})$ for $i\in \var\setminus \ex$ will be called {\em frozen} or {\em coefficients}. By abuse of terminology, we will sometimes simply talk about seeds when we mean graded quantum seeds.

\begin{remark} In contrast to \cite{Gratz-Colimits}, $\ex \subset \var$ is a collection of indices rather than cluster variables.
\end{remark}

Mutation of quantum seeds is defined via certain automorphisms: details are given in \cite[\S 2.5]{GoodearlYakimovQCA}.  We state Corollary~2.6 of \cite{GoodearlYakimovQCA}, which suffices for our purposes.

For $k$ an exchangeable index and $\tilde{B}=(b_{ij})$ an exchange matrix, set
\begin{align*}
[b^{k}]_{+} & = \sum_{b_{ik}>0}b_{ik}e_{i} \qquad \text{and} \\
[b^{k}]_{-} & = \sum_{b_{ik}<0}b_{ik}e_{i}
\end{align*}
Note that the $k$th column of $\tilde{B}$ may be recovered as $\tilde{B}^{k}=[b^{k}]_{+}-[b^{k}]_{-}$.

\begin{corollary}[{\cite[Corollary~2.6]{GoodearlYakimovQCA}}]\label{cor:mutation} For all pairs $(M,\tilde{B})$, with $M$ and $\tilde{B}$ as previously, and for all $k\in \ex$, there exists a toric frame $\mu_{k}(M)\colon \integ^{\var} \to \curly{F}$ such that
\begin{enumerate}[label=\textup{(\alph*)}]
\item the matrix of $\mu_{k}(M)$ is equal to $\mu_{k}(\mathbf{r}(M))$,
\item\label{cor:mutation-exch-relns} the toric frame satisfies
\begin{align*}
\mu_{k}(M)(e_{j}) & =e_{j} \qquad \text{for}\ j\neq k \\
\mu_{k}(M)(e_{k}) & = M(-e_{k}+[b^{k}]_{+})+M(-e_{k}+[b^{k}]_{-}) 
\end{align*}
\item mutation is involutive: $\mu_{k}(\mu_{k}(M,\tilde{B}))=(M,\tilde{B})$. \qed
\end{enumerate}
\end{corollary}

Then for $k\in \ex$, we may define the mutation of a graded quantum seed 
\[
\Sigma=(M,\tilde{B},G,\var,\ex,\inv)
\]
 to be the tuple 
 \[\mu_{k}(\Sigma)=(\mu_{k}(M),\mu_{k}(\tilde{B}),\mu_{k}(G),\var,\ex,\inv)
 \] 
 where each of the first three components has its mutation as defined above and the remaining components are unchanged under mutation.

\subsection{Rooted cluster algebras}\label{S:Rooted cluster algebras associated to triangulations of the closed disc}

We can successively mutate a seed $\Sigma$ along what are called $\Sigma$-admissible sequences. Mutation along all possible $\Sigma$-admissible sequences will provide a prescribed set of generators of the cluster algebra associated to the seed $\Sigma$, the definition of which we will recall in this section.

\begin{definition}[{\cite[Definition~1.3]{ADS}}]\label{D:admissible sequence}
Let $\Sigma = (M,\tilde{B},G,\var,\ex,\inv)$ be a graded quantum seed and $\curly{F}$ the division algebra associated to $M$. For $l \geq 1$ a sequence $(x_1, \ldots, x_l)$ of elements for $\curly{F}$ is called {\em $\Sigma$-admissible} if
\begin{itemize} 
\item $x_1=M(e_{i_{x_{1}}})$ for some $i_{x_{1}} \in \ex$; and
\item for every $2 \leq k \leq l$, we have $x_k=(\mu_{i_{x_{k-1}}}\circ \ldots \circ \mu_{i_{x_{1}}})(M)(e_{i_{x_{k}}})$ for some $i_{x_{k}}\in \ex$.
\end{itemize}
We will call the integer sequence $(i_{x_{1}},\ldots ,i_{x_{k}})$ the {\em mutation index sequence associated to $(x_{1},\ldots ,x_{k})$}. 

The empty sequence of length $l = 0$ is $\Sigma$-admissible for every seed $\Sigma$ and mutation of $\Sigma$ along the empty sequence leaves $\Sigma$ invariant. We denote by
\[
\curly{S}(\Sigma) = \{\mu_{i_{x_{l}}} \circ \ldots \circ \mu_{i_{x_{1}}}(\Sigma) \mid l \geq 0, (x_1, \ldots, x_l) \text{ $\Sigma$-admissible}\}
\]
the set of all graded quantum seeds of $\curly{F}$ which can be reached from $\Sigma$ by iterated mutation along $\Sigma$-admissible sequences and call it the {\em mutation equivalence class of $\Sigma$}.
\end{definition}

Since mutation is involutive, it is clear that mutation along $\Sigma$-admissible sequences induces an equivalence relation on seeds, where two seeds $\Sigma$ and $\Sigma'$ are {\em mutation equivalent} if and only if there exists a $\Sigma $-admissible sequence $(x_1, \ldots, x_l)$ with $\mu_{i_{x_{l}}} \circ \ldots \circ \mu_{i_{x_{1}}}(\Sigma) = \Sigma'$. The mutation class of a seed $\Sigma$ is thus really an equivalence class. Mutation equivalence of exchange matrices is defined analogously.

\begin{remark} \label{R:coefficients in every seed}
It is a direct consequence of the definition of mutation (Corollary~\ref{cor:mutation}\ref{cor:mutation-exch-relns}) that if two seeds $\Sigma$ and $\Sigma'$ are mutation equivalent, then the coefficients of $\Sigma$ are precisely the coefficients of $\Sigma'$. That is, if $\ex$ and $\ex'$ are the exchangeable indices for $\Sigma$ and $\Sigma'$ respectively, then $\ex=\ex'$ and $\var\setminus \ex=\var\setminus \ex'$.  Further, any two mutation equivalent seeds give rise to toric frames with the same division algebra $\curly{F}$ as their codomains.
\end{remark}

Given a graded quantum seed $\Sigma=(M,\tilde{B},G,\var,\ex,\inv)$, consider the mutation equivalence class $\curly{S}(\Sigma)$ of graded quantum seeds of $\curly{F}$ containing $\Sigma$.  Let $\curly{V}(\Sigma)$ be the set of all quantum cluster variables in $\curly{S}(\Sigma)$, that is,
\[ \curly{V}(\Sigma)=\{ M'(e_{i}) \mid (M',\tilde{B}',G',\var,\ex,\inv)\in \curly{S}(\Sigma), i\in \var \}. \]

We also denote by $\cl(\Sigma)$ the set $\{ M(e_{i}) \mid i\in \var \}$, which we call the cluster of $\Sigma$. Set $\cl_{\ex}(\Sigma)=\{ M(e_{i}) \mid i \in \ex \}\subseteq \cl(\Sigma)$.

Let $D$ be a unital subring of $\mathbb{K}$ containing the subgroup generated by the set \[ \{ \Omega_{\mathbf{r}(M)}(f,g) \mid f,g\in \integ^{\var}\}.\]

\begin{definition} The {\em graded quantum cluster algebra}\footnote{We will use $\curly{C}$ for the cluster algebra, rather than $\curly{A}$ as in \cite{ADS} and \cite{Gratz-Colimits}, and reserve $\curly{A}$ for quantum affine spaces.} $\curly{C}(\Sigma)_{D}$ associated to a graded quantum seed $\Sigma=(M,\tilde{B},G,\var,\ex,\inv)$ is the unital $D$-subalgebra of $\curly{F}$ generated by $\curly{V}(\Sigma) \union \{ M(e_{l})^{-1} \mid l\in \mathbf{inv} \}$.  

The {\em rooted graded quantum cluster algebra} with initial seed $\Sigma$ is the pair $(\curly{C}(\Sigma)_{D},\Sigma)$.
\end{definition}

That is, the quantum cluster algebra is generated by all quantum cluster variables in quantum seeds mutation equivalent to $(M,\tilde{B})$, together with the inverses of the frozen variables whose index lies in $\mathbf{inv}$.  Note that while two seeds in the same mutation class give rise to the same cluster algebra, they do not give rise to the same rooted cluster algebra. We can think of rooted cluster algebras as pointed versions of cluster algebras. 

\begin{remark} We have followed \cite{GoodearlYakimovQCA} in stating a definition whereby a specified subset of the frozen variables, the set $\inv \subseteq \var\setminus \ex$, have inverses in the quantum cluster algebra; in the original cluster algebras literature, this subset typically contains all the frozen variables, but in applications this set may be taken to be much smaller or even empty (for example, cluster algebra structures on Grassmannians).
\end{remark}

In this general setting, Goodearl and Yakimov prove that the quantum Laurent phenomenon holds, namely that any quantum cluster variable in $\curly{C}(\Sigma)_{D}$ may be written as a Laurent polynomial in the quantum cluster variables of the initial seed (or indeed any seed).  We refer the reader to \cite[\S 2.5]{GoodearlYakimovQCA} for a precise statement and proof.

In what follows we will always take $D=\mathbb{K}$ and suppress $\mathbb{K}$ in the notation: all of our cluster algebras will be defined over $\mathbb{K}$. From now on, we will also simply say ``cluster algebra'' to refer to a graded quantum cluster algebra, and hence ``rooted cluster algebra'' for the pair $(\curly{C}(\Sigma),\Sigma)$.  As mentioned in Remark \ref{R:classical}, classical commutative cluster algebras and ungraded cluster algebras are simply special cases.

We call the rooted cluster algebra $(\CC(\Sigma),\Sigma)$ {\em skew-symmetric} if the matrix $\tilde{B}$ is skew-sym\-metric (according to the obvious definition in general, $b_{kj}=-b_{jk}$). The {\em rank} of the rooted cluster algebra $(\CC(\Sigma),\Sigma)$ is defined as the cardinality of $\var$.

\begin{remark}
Here we follow \cite{Gratz-Colimits} in our definition of rank.  Traditionally, the rank of a cluster algebra $\CC(\Sigma)$ is defined as the cardinality of the set of exchangeable variables of $\Sigma$, while we define it as the cardinality of the cluster of $\Sigma$. The main interest in this paper is in cluster algebras of infinite rank, and when we talk about those we explicitly want to include cluster algebras associated to seeds with infinitely many coefficients but only finitely many exchangeable variables.
\end{remark}

We also note that the use of the word ``grading'' in Definition~\ref{d:grading} is justified: a $\integ^{I}$-grading in the sense defined there assigns to the cluster variable $M(e_{i})$ the multi-degree $G_{i}$ and this extends via mutation to an algebra $\integ^{I}$-grading on $\CC(\Sigma)$ for which every cluster variable is homogeneous.  We refer the reader to \cite{GradedClusterAlgebras} for a more detailed discussion of cluster algebra gradings.

\section[Rooted cluster morphisms and the category of rooted cluster algebras]{Rooted cluster morphisms and the category of rooted cluster algebras
\sectionmark{Rooted cluster morphisms}
}\label{S:Category of rooted cluster algebras}
\sectionmark{Rooted cluster morphisms}

\subsection{Rooted cluster morphisms}

\begin{definition}[{\cite[Definition~2.1]{ADS}}]\label{D:biadmissible}
Let $\Sigma$ and $\Sigma'$ be seeds and let $f \colon \CC(\Sigma) \to \CC(\Sigma')$ be a map between their associated cluster algebras. A $\Sigma$-admissible sequence $(x_1, \ldots, x_l)$ whose image $(f(x_1), \ldots, f(x_l))$ is $\Sigma'$-admis\-sible is called {\em $(f,\Sigma,\Sigma')$-biadmissible}. 
\end{definition}

\begin{definition}[{\cite[Definition~2.2]{ADS}}]\label{D:rooted cluster morphism}
Let \[\Sigma = (M,\tilde{B},G,\var,\ex,\inv) \quad \text{and} \quad \Sigma'=(M',\tilde{B}',G',\var',\ex',\inv')\] be seeds and let $(\CC(\Sigma),\Sigma)$ and $(\CC(\Sigma'),\Sigma')$ be the corresponding rooted cluster algebras.

A {\em rooted cluster morphism} from $(\CC(\Sigma),\Sigma)$ to $(\CC(\Sigma'),\Sigma')$ is a graded $\mathbb{K}$-algebra homomorphism $f \colon \CC(\Sigma) \to \CC(\Sigma')$ satisfying the following conditions:
\begin{enumerate}[label=(CM\arabic*)]
\item\label{CM1} $f(\cl(\Sigma)) \subseteq \cl(\Sigma') \union \mathbb{K}$.
\item\label{CM2} $f(\cl_{\ex}(\Sigma)) \subseteq \cl_{\ex'}(\Sigma') \union \mathbb{K}$.
\item\label{CM3} the homomorphism $f$ commutes with mutation along $(f,\Sigma,\Sigma')$\--bi\-ad\-mis\-sible sequences, i.e.\ for every $(f,\Sigma,\Sigma')$-biadmissible sequence $(x_1,\ldots, x_l)$ we have
\[
f(\mu_{i_{x_l}} \circ \ldots \circ \mu_{i_{x_1}}(y)) = \mu_{i_{f(x_l)}} \circ \ldots \circ \mu_{i_{f(x_1)}}(f(y))
\]
for all $y \in \cl(\Sigma)$ with $f(y) \in \cl(\Sigma')$.
\end{enumerate}
\end{definition}

Some comments on the definition are required.  Firstly, given a rooted cluster morphism, we have an induced function $\overline{f}\colon \var \to \overline{\var'}$ where $\overline{\var'}=\var' \disjointunion \{ \infty \}$ given by
\[ \overline{f}(i)=\begin{cases} j & \text{if}\ f(M(e_{i}))=M'(e_{j}) \\ \infty & \text{if}\ f(M(e_{i})) \in \mathbb{K} \end{cases} \]
Note that \emph{a priori}, $\overline{f}$ need not be injective or surjective.

Next, by graded algebra homomorphism, we mean a graded homomorphism of degree zero, i.e.\ $f(\CC(\Sigma)_{a})\subseteq \CC(\Sigma')_{a}$ for all $a\in \integ^{I}$, where $\CC(\Sigma)_{a}$ denotes the homogeneous component of $\CC(\Sigma)$ of degree $a$.  In particular, if $f(M(e_{i}))\in \cl(\Sigma')$, the degree of $f(M(e_{i}))$ must be the same as the degree of $M(e_{i})$.  

Since elements of $\mathbb{K}$ are of degree zero, this also implies that only elements of $\cl(\Sigma)$ that have degree zero can be mapped to $\mathbb{K}$ under $f$.  That is, $G_{\overline{f}^{-1}(j)}=G'_{j}$ for all $j\in \im \overline{f}\subseteq \var'$ and if $\overline{f}(i)=\infty$, we must have $G_{i}=0$.

Secondly, note that the condition that $f$ is an algebra homomorphism is very restrictive in the noncommutative setting: it immediately implies that for all $i,j\in \overline{f}^{-1}(\var')$, the cluster variables $f(M(e_{i}))$ and $f(M(e_{j}))$ have the same quasi-commutation relation as $M(e_{i})$ and $M(e_{j})$, i.e.\ $\mathbf{r}(M)_{\overline{f}^{-1}(j)}=\mathbf{r}(M')_{j}$ for all $j\in \im \overline{f}$.  Again, if $f(M(e_{i}))\in \cl(\Sigma')$, this has strong implications for $M(e_{i})$.

\begin{definition}[\cite{GrabowskiLaunois-Lifting}]
Let $\mathbf{q}$ be a multiplicatively skew-symmetric matrix.  A subset $\mathbf{z} \subseteq \var$ such that for all $z\in \mathbf{z}$, $q_{zj}=q_{jz}=1$ for all $j\in \var$ will be called a {\em $\mathbf{q}$-central subset} of $\var$.
\end{definition}

Given a $\mathbf{q}$-central subset $\mathbf{z}$, the centre of the quantum torus $\curly{T}_{\mathbf{q}}$ contains the subalgebra generated by $\{ Y_{z}^{\pm 1} \mid z\in \mathbf{z}\}$, hence the name.  Examining the condition for compatible pairs above, we see that a mutable variable cannot be $\mathbf{r}(M)$-central: if $M(e_{k})$ is $\mathbf{r}(M)$-central then $t_{kk}=1$ for any exchange matrix $\tilde{B}$, which is not permitted.

Then since elements of $\mathbb{K}$ are central, if $f(M(e_{i}))\in \mathbb{K}$ we must have that $M(e_{i})$ is central and so is not mutable.  Hence the following observation:

\begin{lemma} Let $f$ be a rooted cluster morphism as above.  Then $f^{-1}(\infty)\subseteq \var\setminus \ex$ is $\mathbf{r}(M)$-central.
\end{lemma}

Note however that $M(e_{i})$ being $\mathbf{r}(M)$-central does not necessarily imply $f(i)=\infty$; if $M(e_{i})$ is $\mathbf{r}(M)$-central but $f(i)\neq \infty$, we simply deduce that $M'(e_{f(i)})$ is $\mathbf{r}(M')$-central.  That is, we may have central coefficients whose image is not in $\mathbb{K}$.

Then in the definition of a rooted cluster morphism, we may replace \ref{CM2} with
\begin{enumerate}[label=(CM\arabic*$'$),start=2]
\item\label{CM2prime} $f(\cl_{\ex}(\Sigma)) \subseteq \cl_{\ex'}(\Sigma')$.
\end{enumerate}

The above discussion shows that in fact \ref{CM2prime} is not a stronger assumption.

From now on by abuse of notation we write $\CC(\Sigma)$ for the rooted cluster algebra $(\CC(\Sigma), \Sigma)$. 

\begin{example}\label{E:rooted cluster morphism}
Let $q\in \mathbb{K}^{*}$.  Consider the following seeds:
\begin{itemize}
\item $\Sigma = (M,\tilde{B},G,\var,\ex,\inv)$, with
\begin{align*} \mathbf{r}(M) & = \begin{pmatrix} 1 & q & 1 \\ q^{-1} & 1 & 1 \\ 1 & 1 & 1 \end{pmatrix} & \tilde{B} & = \begin{pmatrix} 0 & 1 & 0 \\ -1 & 0 & -1 \\ 0 & 1 & 0 \end{pmatrix} & G & =\begin{pmatrix} 0 & 1 & 0 \end{pmatrix} \\ \var & = \{ 1,2,3 \} & \ex & = \{ 2 \} & \inv & = \emptyset \\
\end{align*}
Set $x_{i}=M(e_{i})$.  Then, associating to the skew-symmetric matrix $\tilde{B}$ a quiver in the usual way, we may represent the initial cluster as
\[ \begin{tikzpicture}[baseline={([yshift={-\ht\strutbox}]current bounding box.north)}, scale=2.5,cap=round,>=latex,on grid,node distance=2cm]
        \tikzstyle{every node}=[font=\small]
\node[rectangle, draw] (1) at (0,0) {$x_1$};
\node[right of=1] (2) {$x_2$};
\node[right of=2,rectangle,draw] (3) {$x_3$};
\node[above of=1,yshift=-1.5cm] {$0$};
\node[above of=2,yshift=-1.5cm] {$1$};
\node[above of=3,yshift=-1.5cm] {$0$};
\draw[->] (1) -- (2);
\draw[->] (3) -- (2);
\end{tikzpicture}
\]
where we indicate the degree of the element by the integers over each vertex.
\item $\Sigma' = (M',\tilde{B}',G',\var',\ex',\inv')$, with
\begin{align*} \mathbf{r}(M') & = \begin{pmatrix} 1 & 1 & 1 \\ 1 & 1 & q \\ 1 & q^{-1} & 1 \end{pmatrix} & \tilde{B'} & = \begin{pmatrix} 0 & 1 & 0 \\ -1 & 0 & 1 \\ 0 & -1 & 0 \end{pmatrix} & G' & =\begin{pmatrix} 1 & 0 & 1 \end{pmatrix} \\ \var' & = \{ 0,1,2 \} & \ex' & = \{ 1,2 \} & \inv' & = \emptyset \\
\end{align*}
Set $y_{i}=M'(e_{i})$.  Then we may represent the initial cluster as
\[ \begin{tikzpicture}[baseline={([yshift={-\ht\strutbox}]current bounding box.north)}, scale=2.5,cap=round,>=latex,on grid,node distance=2cm]
        \tikzstyle{every node}=[font=\small]
\node[rectangle, draw] (1) at (0,0) {$y_0$};
\node[right of=1] (2) {$y_1$};
\node[right of=2] (3) {$y_2$};
\node[above of=1,yshift=-1.5cm] {$1$};
\node[above of=2,yshift=-1.5cm] {$0$};
\node[above of=3,yshift=-1.5cm] {$1$};
\draw[->] (1) -- (2);
\draw[->] (2) -- (3);
\end{tikzpicture}
\]
\end{itemize}
As an example of the calculation of quantum cluster variables, we have that
\begin{align*}
\mu_{2}(x_{2}) & = \curly{S}_{\mathbf{r}(M)}(1,-1,1)x^{(1,-1,1)}+\curly{S}_{\mathbf{r}(M)}(0,-1,0)x^{(0,-1,0)} \\
 & = (r_{12}r_{13}^{-1}r_{23})x_{1}x_{2}^{-1}x_{3}+x_{2}^{-1} \\
 & = qx_{1}x_{2}^{-1}x_{3}+x_{2}^{-1}
\end{align*}
The rooted cluster morphism $f\colon \curly{C}(\Sigma)\to \curly{C}(\Sigma')$ we consider is defined on the initial cluster variables as
\[ x_{1} \mapsto y_{1}, \qquad x_{2} \mapsto y_{2}, \qquad x_{3} \mapsto 1. \]
That is,
\[ \begin{tikzpicture}[baseline={([yshift={-\ht\strutbox}]current bounding box.north)}, scale=2.5,cap=round,>=latex,on grid,node distance=2cm]
        \tikzstyle{every node}=[font=\small]
\node[rectangle, draw] (1) at (0,0) {$x_1$};
\node[right of=1] (2) {$x_2$};
\node[right of=2,rectangle,draw] (3) {$x_3$};
\node[above of=1,yshift=-1.5cm] {$0$};
\node[above of=2,yshift=-1.5cm] {$1$};
\node[above of=3,yshift=-1.5cm] {$0$};
\draw[->] (1) -- (2);
\draw[->] (3) -- (2);

\node[below of=1] (2b) {$y_1$};
\node[left of=2b,rectangle, draw] (1b) {$y_0$};
\node[right of=2b] (3b) {$y_2$};
\node[right of=3b] (4b) {$1$};
\node[below of=1b,yshift=1.5cm] {$1$};
\node[below of=2b,yshift=1.5cm] {$0$};
\node[below of=3b,yshift=1.5cm] {$1$};
\node[below of=4b,yshift=1.5cm] {$0$};
\draw[->] (1b) -- (2b);
\draw[->] (2b) -- (3b);

\draw[->,shorten <=0.5em,dashed] (1) -- (2b);
\draw[->,shorten <=0.5em,dashed] (2) -- (3b);
\draw[->,shorten <=0.5em,dashed] (3) -- (4b);
\end{tikzpicture}
\]
This extends to a homomorphism of the associated rooted graded quantum cluster algebras and we see that it is also graded of degree $0$.  Notice that $x_{3}$ is $\mathbf{r}(M)$-central (and $3\notin \ex$) and $f(x_{3})=1$ is also central.  We also visibly have that $f$ satisfies \ref{CM1} and \ref{CM2prime}.

The induced function on $\var$ is $\overline{f}\colon \var \to \overline{\var'}$, $\overline{f}(1)=1$, $\overline{f}(2)=2$ and $\overline{f}(3)=\infty$.

To verify \ref{CM3}, notice that the only exchangeable cluster variable in $\Sigma$ whose image is exchangeable in $\Sigma'$ is $x_2$ with $f(x_2)=y_2$, so the first entry of every $(f,\Sigma,\Sigma')$-biadmissible sequence has to be $x_2$. 
Via the above calculation and a similar one for $\mu_{2}(y_{2})$, we have
\[
f(\mu_{2}(x_2)) = f(qx_1x_{2}^{-1}x_3 + x_2^{-1}) = qy_1y_{2}^{-1}+y_2^{-1} = \mu_{2}(y_2) = \mu_{f(2)}(f(x_2))
\]
and, since $f(x_i) \neq f(x_2)$ for $i = 1,3$, we have \[f(\mu_{2}(x_i)) = f(x_i) = \mu_{f(2)}(f(x_i))\] for these values of $i$.
Furthermore, the only exchangeable cluster variable in $\mu_{2}(\Sigma)$ whose image is exchangeable in $\mu_{2}(\Sigma')$ is  $\mu_{2}(x_2)$ with $f(\mu_{2}(x_2)) = \mu_{2}(y_2)$, so all $(f,\Sigma,\Sigma')$-biadmissible sequences have alternating entries $x_2$ and $\mu_{2}(x_2)$. Since mutation is involutive, the ring homomorphism $f$ commutes with mutation along any of these sequences. Thus axiom \ref{CM3} is satisfied and $f$ is a rooted cluster morphism.
\end{example}

\subsection{The category of rooted cluster algebras}
Considering rooted cluster algebras and rooted cluster morphisms gives rise to a category.  Recall that by convention by ``rooted cluster'' we mean ``rooted graded quantum cluster''; by adopting this convention, the next definition of \cite{ADS} can be given as a verbatim statement, interpreted in our more general setting.

\begin{definition}[{\cite[Definition~2.6]{ADS}}]
The {\em category of rooted cluster algebras $\Clus$} is the category which has as objects rooted cluster algebras and as morphisms rooted cluster morphisms.
\end{definition}

In \cite[Section~2]{ADS} it was shown that $\Clus$ satisfies the axioms of a category, the key point being the composibility of rooted cluster morphisms.  The argument of \cite[Proposition~2.5]{ADS} carries over identically to the graded quantum setting.

\subsection{Coproducts and connectedness of seeds} \label{ss:Connectedness of seeds and coproducts}

Assem, Dupont and Schiffler showed in \cite[Lemma~5.1]{ADS} that countable coproducts exist in the category $\Clus$ of rooted cluster algebras.  Following their proof, we obtain the corresponding result in our setting, except with arbitrary coproducts.

\pagebreak
\begin{lemma} The category $\Clus$ admits all coproducts.
\end{lemma}

\begin{proof} Let $I$ be a set and $\{ \CC(\Sigma_{i}) \mid i\in I \}$ a family of rooted cluster algebras.  For any $i\in I$, we have $\Sigma_{i}=(M_{i},\tilde{B}_{i},G_{i},\var_{i},\ex_{i},\inv_{i})$.  Here $M_{i}\colon \integ^{\var_{i}} \to \curly{F}_{i}$ with $\curly{F}_{i}$ the skew-field of fractions of a quantum torus $\curly{T}_{i}=\curly{T}_{\mathbf{r}(M_{i})^{\cdot 2}}$.  Also, each $G_{i}$ is a $\var_{i} \cross \mathbf{D}_{i}$ matrix, for some set $\mathbf{D}_{i}$.  

We define a new seed $\Sigma=(M,\tilde{B},G,\var,\ex,\inv)$ by
\begin{itemize}
\item $\var=\bigdisjointunion_{i\in I} \var_{i}$;
\item $\ex=\bigdisjointunion_{i\in I} \ex_{i}$;
\item $\inv=\bigdisjointunion_{i\in I} \inv_{i}$;
\item $M\colon \integ^{\var}\to \curly{F}$ with $\curly{F}$ the skew-field of fractions of $\bigtensor_{i\in I} \curly{T}_{i}$, defined by $\mathbf{r}(M)_{jk}=\mathbf{r}(M_{l})_{jk}$ if $j,k\in \var_{l}$, and 1 otherwise;
\item $\tilde{B}$ the $(\var \cross \ex)$-indexed matrix with $\tilde{B}_{jk}=(\tilde{B}_{l})_{jk}$ if $j\in \var_{l}$ and $k\in \ex_{l}$, and $0$ otherwise;
\item $G$ the $\var \cross \left(\bigdisjointunion \mathbf{D}_{i}\right)$-indexed matrix with $G_{jk}=(G_{l})_{jk}$ if $j\in \var_{l}$ and $k\in \mathbf{D}_{l}$, and 0 otherwise.
\end{itemize}
Note that the last two definitions are the natural generalisation to arbitrary indexing sets of the direct sum of matrices, placing the $\tilde{B}_{i}$ (respectively $G_{i}$) on the diagonal of $\tilde{B}$ (resp.\ $G$) and completing to a matrix by zeroes elsewhere.  The definition of $\mathbf{r}(M)$ is similar but is the multiplicatively skew-symmetric analogue of this.

Via the natural isomorphism $\integ^{\var}\iso \bigdsum_{i\in I} \integ^{\var_{i}}$, it is straightforward to see that $M$ is a toric frame.  Also $\tilde{B}$ is locally finite, as every $\tilde{B}_{i}$ is, and compatibility of $M$ and $\tilde{B}$ follows immediately from the compatibility of $M_{i}$ and $\tilde{B}_{i}$ for all $i$.  Similarly, $G$ is a grading for $\tilde{B}$ and hence $\Sigma$ is a well-defined graded quantum seed.

For any $i\in I$, we have a canonical inclusion $j_{i}\colon \curly{F}_{i}\to \curly{F}$, as the $i$th tensor factor, and clearly this induces a rooted cluster morphism $\CC(\Sigma_{i}) \to \CC(\Sigma)$.

The remainder of the proof exactly follows the argument of \cite{ADS}.  That is, if $\CC(\Theta)$ is a rooted cluster algebra and for any $i\in I$ we have a rooted cluster morphism $g_{i}\colon \CC(\Sigma_{i}) \to \CC(\Theta)$, there exists exactly one graded ring homomorphism $h\colon \curly{F}_{\Sigma} \to \curly{F}_{\Theta}$ satisfying $h(x)=g_{i}(x)$, since tensor products are coproducts. Hence there exists exactly one rooted cluster morphism $h\colon \CC(\Sigma) \to \CC(\Theta)$ satisfying $h\circ j_{i} =g_{i}$ for all $i\in I$ and we are done.
\end{proof}

Taking coproducts of a family $\{\CC(\Sigma_i)\}_{i \in I}$ of rooted cluster algebras amounts to taking what can be intuitively described as the disjoint union $\Sigma$ of their seeds. The seeds $\Sigma_i$ will be full subseeds of the seed $\Sigma$ which are mutually disconnected.  We refer to \cite[Section~3.4]{Gratz-Colimits} for a more in depth discussion of connectedness of seeds and the relationship with coproducts, which also applies in our more general setting. In general, we say that a seed $\Sigma = (M,\tilde{B},G,\var,\ex,\inv)$ is {\em connected}, if for all $k \neq l \in \var$ there exists a finite sequence $i_0, \ldots, i_j \in \var$ such that $k = i_0$, $l = i_j$ and for all $1 \leq m \leq j$ we have
$\tilde{B}_{i_{m}i_{m+1}} \neq 0$.

In particular, as discussed in \cite{Gratz-Colimits}, every connected component of a seed has a countable cluster.  As a consequence, one does not currently gain much from considering uncountable clusters. The usual operations on seeds, namely mutations along (finite) admissible sequences, affect only finitely many connected components and hence only operate on a countable full subseed which is not connected to its invariant complement. So for all practical purposes one can restrict to working with countable seeds without any substantial loss of generality.

\subsection{Monomorphisms and epimorphisms in \texorpdfstring{$\Clus$}{Clus}}

There is a substantial discussion of both monomorphisms and epimorphisms in the category of classical rooted cluster algebras in \cite{ADS}.  We expect that many of the statements, and indeed proofs, can be carried over to the graded quantum setting.  However, we shall not do this here, as our principal goal is the claim of the title.

We note, though, that in \cite{GrabowskiLaunois-Lifting} the first author and Launois prove the following, which is a generalisation of the notion of simple specialisation of \cite[\S 6.1]{ADS}.

Let $\curly{C}(M,\tilde{B},G,\var,\mathbf{ex},\mathbf{inv})_{D}$ be a graded quantum cluster algebra and let $\mathbf{r} = \mathbf{r}(M)$ be the matrix associated to the toric frame $M$.  

Let $\mathbf{z}\subseteq \var\setminus \ex$ be $\mathbf{r}$-central and let $\integ^{\var\setminus \mathbf{z}}$ denote the free Abelian group with basis $\{ e_{j} \mid j\in \var\setminus \mathbf{z} \}$.  Let $\iota\colon \integ^{\var\setminus \mathbf{z}}\to \integ^{\var}$ be the associated natural inclusion.

Set $\bar{\curly{F}}=\mathrm{Fract}(\curly{T}_{\bar{\mathbf{r}}^{\cdot 2}})$, where $\curly{T}_{\bar{\mathbf{r}}^{\cdot 2}}=\curly{T}_{\overline{\mathbf{r}^{\cdot 2}}}$ is defined as the toric frame with respect to $\overline{\mathbf{r}^{\cdot 2}}=\mathbf{r}^{\cdot 2}|_{\var\setminus \mathbf{z}}$.  

Then $\bar{M}\colon \integ^{\var\setminus \mathbf{z}} \to \bar{\curly{F}}$, $\bar{M}=M\circ \iota$ is a toric frame with associated matrix $\mathbf{r}(\bar{M})=\bar{\mathbf{r}}$.  Furthermore, $(\bar{B}=\tilde{B}|_{\var\setminus \mathbf{z}}^{\var\setminus \mathbf{z}},\bar{\mathbf{r}})$ is a compatible pair and the pair $(\bar{M},\bar{B})$ is a quantum seed of $\bar{\curly{F}}$.  

Assume that $\{ Y_{z} \mid z\in \mathbf{z} \}$ are all of degree zero with respect to $G$; set $\bar{G}=G|_{\var \setminus \mathbf{z}}$.

\begin{proposition}
The map 
\[ \pi_{\curly{C}} \colon \curly{C}(M,\tilde{B},G,\var,\mathbf{ex},\mathbf{inv})_{D} \to \curly{C}(\bar{M},\bar{B},\bar{G},\var\setminus \mathbf{z},\mathbf{ex},\mathbf{inv} \setminus \mathbf{z})_{D}
 \]
 defined by 
\[ Y_{i} \mapsto \begin{cases} \bar{Y}_{i} & \text{if}\ i\in \var\setminus \mathbf{z} \\ 1 & \text{if}\ i\in \mathbf{z} \end{cases}
\]
is a surjective algebra homomorphism.
\end{proposition}

In the nomenclature of \cite{ADS}, $\pi_{\curly{C}}$ is an ideal surjective rooted cluster morphism.

The condition that $\mathbf{z}\subseteq \var\setminus \ex$, i.e.\ consists of frozen indices, can likely be relaxed, cf.\ \cite[Corollary~6.4]{ADS}.
  
\section[Rooted cluster algebras of infinite rank as colimits of rooted cluster algebras of finite rank]{Rooted cluster algebras of infinite rank as colimits of rooted cluster algebras of finite rank
\sectionmark{Rooted cluster algebras of infinite rank}
}\label{S:colimits}
\sectionmark{Rooted cluster algebras of infinite rank}

In this section, we show that every rooted cluster algebra of infinite rank can be written as a linear colimit of rooted cluster algebras of finite rank. This yields a formal way to manipulate cluster algebras of infinite rank by viewing them locally as cluster algebras of finite rank. 

\subsection{Colimits and limits in \texorpdfstring{$\Clus$}{Clus}}

We start by recalling the definition of colimits and limits (their dual notion). Let $\CC$ and $\J$ be categories and let $F \colon \J \to \CC$ be a diagram of type $\J$ in the category $\CC$, i.e.\ a functor from $\J$ to $\CC$. 

The {\em colimit $\mathrm{colim} (F)$ of $F$} (if it exists) is an object $\mathrm{colim} (F) \in \CC$ together with a family of morphisms $f_i\colon F(i) \to \mathrm{colim} (F)$ in $\CC$ indexed by the objects $i \in \J$ such that for any morphism $f_{ij}\colon i \to j$ in $\J$ we have $f_j \circ F(f_{ij}) = f_i$ and for any object $C \in \CC$ with a family of morphisms $g_i\colon F(i) \to C$ in $\CC$ for objects $i \in \J$  such that $g_j \circ F(f_{ij}) = g_i$ for all morphisms $f_{ij}\colon i \to j$ in $\J$, there exists a unique morphism $h \colon \mathrm{colim} (F) \to C$ such that the following diagram commutes.
\[
\xymatrix{& C  & \\ & \mathrm{colim}( F) \ar[u]^h & \\ F(i) \ar[ruu]^{g_i} \ar[ru]_{f_i} \ar[rr]^{F(f_{ij})} && F(j) \ar[luu]_{g_j} \ar[lu]^{f_j}}
\]
The {\em limit $\mathrm{lim} (F)$ of $F$} (if it exists) is defined dually.

A category is called {\em complete}, respectively {\em cocomplete}, if it has all small limits, respectively colimits.

\subsection{Rooted cluster algebras of infinite rank as colimits} 
We will now show our main result that there are sufficient colimits such that every graded quantum rooted cluster algebra of infinite rank is isomorphic to a colimit of graded quantum rooted cluster algebras of finite rank. 

When the initial seed of an infinite rank rooted cluster algebra is connected, we can even write it as a linear colimit. A colimit $\mathrm{colim}(F)$ in a category $\CC$ is called {\em linear} if the index category $\J$ of the diagram $F \colon \J \to \CC$ is a set endowed with a linear order viewed as a category. A diagram $F \colon \J \to \CC$ where $\J$ is endowed with a linear order $\leq$ is just a {\em linear system} of objects in $\CC$, that is a family of objects $\{C_i\}_{i \in \J}$ and a family of morphisms $\{ f_{ij}\}_{i  \leq j \in \J}$ such that $f_{jk} \circ f_{ij} = f_{ik}$ and $f_{ii} = \id_{C_i}$ for all $i\leq j \leq k$ in $\J$. In order to explicitly construct a suitable linear system of rooted cluster algebras of finite rank, we use the fact that in certain nice cases inclusions of subseeds give rise to rooted cluster morphisms.

\begin{definition}
	Let $\Sigma = (M, \tilde{B}, G, \var, \ex, \inv)$ in $\curly{F}$ and $\Sigma' = (M', \tilde{B}', G', \var', \ex', \inv')$ in $\curly{F}'$ be graded quantum seeds. We say that $\Sigma$ is a {\em full subseed} of $\Sigma'$ if
		\begin{enumerate}
			\item $\var \subseteq \var'$, $\ex \subseteq \ex'$ and $\inv \subseteq \inv'$
			\item $\curly{F} \subseteq \curly{F}'$
			\item $\tilde{B}' \mid_{\var}^{\var} = \tilde{B}$
			\item $M'(a) = M(a)$ for all $a \in \ZZ^{\var'}$ with $\supp(a) \subseteq \var$.
			\item $G'$ is a $(\var' \times I)$-grading and $G$ is a $(\var \times I)$-grading with $G' \mid_{\var} = G$.
		\end{enumerate}
	If $\Sigma$ is a full subseed of $\Sigma'$, we say that $\Sigma$ and $\Sigma'$ are connected only by coefficients in $\Sigma$ if additionally $\tilde{B}'_{ij} = 0$ for all $i \in \ex$ and $j \in \var' \setminus \var$.  
\end{definition}

\begin{remark}
	The definition of full subseed is in line with \cite{Gratz-Colimits}, but is slightly more general than the one given in \cite{ADS}: We allow frozen variables in $\Sigma$ to be exchangeable in $\Sigma'$, whereas in \cite{ADS}, 
	a seed  $\Sigma$ is only called a full subseed of $\Sigma'$ if all of the above conditions for a full subseed are met and additionally all frozen variables in $\Sigma$ are frozen variables in $\Sigma'$.
\end{remark}

\begin{remark} To check the condition of being connected only by coefficients, particularly for quivers, the contrapositive statement is perhaps more instructive: $\Sigma$ and $\Sigma'$ are connected only by coefficients if for all $i\in \ex$, if $\tilde{B}_{ij}\neq 0$ then $j\in \var$.  That is, looking at a mutable vertex $i$ in the quiver for $\Sigma$, viewed as a full subquiver of the quiver for $\Sigma'$, all arrows to or from $i$ should be to vertices that already belonged to $\Sigma$.
\end{remark}

\begin{lemma}\label{L:full subseeds}
	Let 
		\[
			\Sigma = (M, \tilde{B}, G, \var, \ex, \inv)
		\]
	in $\curly{F}$ be a full subseed of 
		\[
			\Sigma' = (M', \tilde{B}', G', \var', \ex', \inv')
		\]
	in $\curly{F}'$ and let them be connected only by coefficients in $\Sigma$. Let $(x_1, \ldots, x_l)$
	be a $\Sigma$-admissible sequence. Then $(x_1, \ldots, x_l)$ is a $\Sigma'$-admissible sequence and $\tilde{\Sigma} = \mu_{i_{x_l}} \circ \ldots \circ \mu_{i_{x_1}}(\Sigma)$ is a full subseed of $\tilde{\Sigma}' = 
	\mu_{i_{x_l}} \circ \ldots \circ \mu_{i_{x_1}}(\Sigma')$, where $\tilde{\Sigma}$ and $\tilde{\Sigma}'$ are only connected by coefficients in $\tilde{\Sigma}$.
\end{lemma}

\begin{proof}
	We prove the claim by induction over the length $l$ of the $\Sigma$-admissible sequence \\$(x_1, \ldots, x_l)$. The statement is true for $l = 0$ by assumption. Assume now the statement is true for all $\Sigma$-admissible
	sequences of length $l$ and consider a $\Sigma$-admissible sequence $(x_1, \ldots, x_{l+1})$. We have $x_{l+1} = \mu_{i_{x_l}} \circ \ldots \circ \mu_{i_{x_1}}(M)(e_{i_{x_{l+1}}})$ for some $i_{x_{l+1}} \in \ex$. By the inductive hypothesis, we have $x_{l+1} = \mu_{i_{x_l}} \circ \ldots \circ \mu_{i_{x_1}}(M)(e_{i_{x_{l+1}}}) = \mu_{i_{x_l}} \circ \ldots \circ \mu_{i_{x_1}}(M')(e_{i_{x_{l+1}}})$ and thus $(x_1, \ldots, x_{l+1})$ is $\Sigma'$-admissible.
	
	Set now $\mu_{i_{x_l}} \circ \ldots \circ \mu_{i_{x_1}}(\Sigma) = \Sigma^l = (M^l,\tilde{B}^l,G^l, \var, \ex, \inv)$ and $\mu_{i_{x_l}} \circ \ldots \circ \mu_{i_{x_1}}(\Sigma') = \Sigma'^l = (M'^l,\tilde{B}'^l,G'^l, \var', \ex', \inv')$. 
	By the inductive hypothesis $\Sigma^l$ is a full subseed of $\Sigma'^l$ and they are only connected by coefficients in $\Sigma^l$.

	We have 
		\[
			\tilde{\Sigma} = (\mu_{i_{x_{l+1}}}(M^l),\mu_{i_{x_{l+1}}}(\tilde{B}^l),\mu_{i_{x_{l+1}}}(G^l), \var, \ex, \inv)
		\] 
	and 
		\[
			\tilde{\Sigma}' = (\mu_{i_{x_{l+1}}}(M'^l),\mu_{i_{x_{l+1}}}(\tilde{B}'^l),\mu_{i_{x_{l+1}}}(G'^l), \var', \ex', \inv').
		\] 
	
	We have $\var \subseteq \var'$, $\ex \subseteq \ex'$ and $\inv \subseteq \inv'$ and $\curly{F} \subseteq \curly{F}'$ by assumption. For notational convenience set $k = i_{x_{l+1}}$.
	
	Consider first $\mu_k(\tilde{B}'^l)$: We have
		\[
			\mu_k(\tilde{B}'^l)_{ij} = 	\begin{cases}
									-b'^l_{ij} \text{ if } i = k \text{ or } j = k \\
									b'^l_{ij} + \frac{1}{2}(|b'^l_{ik}|b'^l_{kj} + b'^l_{ik}|b'^l_{kj}|), \text{ otherwise.}
								\end{cases}
		\]
	Comparing with $\mu_k(\tilde{B}^l)$ yields $\mu_k(\tilde{B}'^l)_{ij} = \mu_k(\tilde{B}^l)_{ij}$ if $i,j \in \var$. If $i \in \ex$ and $j \in \var' \setminus \var$, then the inductive hypothesis (bearing in mind that $k \in \ex$) yields
	$\mu_k(\tilde{B}'^l)_{ij} = 0$.
	
	Consider now $\mu_k(M'^l)$: We have $\mu_k(M'^l) (e_i) = M'^l(e_i) = M^l(e_i)$ for all $i \in \var \setminus \{k\}$. For $i = k$ we get
\[ \mu_k(M'^l) (e_k) = M'^l(-e_k + [(b'^l)^k]_+) + M'^l(-e_k + [(b'^l)^k]_-).\]  By the inductive hypothesis 	on the exchange matrices we have $-e_k + [(b'^l)^k]_+ = -e_k + [(b^l)^k]_+$\linebreak and $-e_k + [(b'^l)^k]_- = -e_k + [(b^l)^k]_-$ and thus $\supp(-e_k + [(b'^l)^k]_+) \subseteq \var$ and\linebreak $\supp(-e_k + [(b'^l)^k]_-) \subseteq \var$. By the inductive hypothesis on the toric frames we have
\[ \mu_k(M'^l) (e_k) = M^l(-e_k + [(b^l)^k]_+) + M^l(-e_k + [(b^l)^k]_-) = \mu_k(M^l)(e_k).\] Since the values of the entries in the matrices $\mathbf{r}(M')\mid_{\var}^{\var}$ and $\mathbf{r}(M)$ are equal, with $\mathbf{r}(M)_{ij}=M(e_{i})M(e_{j})M(e_{i}+e_{j})^{-1}=\mathbf{r}(M')_{ij}$ for $i,j\in \var$, we have $\curly{S}_{\mathbf{r}(M')}(a) = \curly{S}_{\mathbf{r}(M)}(a)$ for all $a \in \mathbb{Z}^{\var'}$ with $\supp(a) 		\subseteq \var$. Thus the toric frame $M'$ takes the same values as $M$ on all elements of $\mathbb{Z}^{\var'}$ with support in $\var$.
	
	Finally, consider $\mu_k(G'^l)$: It is straightforward to check (cf.\ for example Section 3 in \cite{GradedClusterAlgebras}) that for the rows of $\mu_k(G'^l)$ we have $\mu_k(G'^l)_i = G'^l_i$ if $i \neq k$ and 
	$\mu_k(G'^l)_k = (-e_k - [(b'^l)^k]_-)^T G$ and similarly, $\mu_k(G^l)_i = G^l_i$ if $i \neq k$ and $\mu_k(G^l)_k = (-e_k - [(b^l)^k]_-)^T G$. By the the inductive hypothesis on the exchange matrices it follows directly
	that $\mu_k(G'^l) \mid_{\var} = \mu_k(G^l)$.
	\end{proof}

In general, if $\Sigma$ is a full subseed of $\Sigma'$, the natural inclusion $\curly{F} \to \curly{F}'$ does not give rise to a rooted cluster morphism $\CC(\Sigma) \to \CC(\Sigma')$, even in the classical setting (see \cite[Remark~4.10]{ADS}). However, it does if $\Sigma$ and $\Sigma'$ are connected only by coefficients in $\Sigma$ .

\begin{theorem}\label{T:inclusion morphism}
	Let 
		\[
			\Sigma = (M, \tilde{B}, G, \var, \ex, \inv)
		\]
	in $\curly{F}$ be a full subseed of 
		\[
			\Sigma' = (M', \tilde{B}', G', \var', \ex', \inv')
		\]
	in $\curly{F}'$, connected only by coefficients in $\Sigma$. Then the natural inclusion $f \colon \curly{F} \to \curly{F}'$ restricts to a rooted cluster morphism $f \colon \mathcal{C}(\Sigma) \to \mathcal{C}(\Sigma')$.
\end{theorem}

\begin{proof}
	By Lemma \ref{L:full subseeds}, every quantum cluster variable $x \in \curly{V}(\Sigma)$ is a quantum cluster variable in $\curly{V}(\Sigma')$ and 
	$f(x) = x$. Furthermore, since $\inv \subseteq \inv'$, we have 
\[ \{M(e_l)^{-1} \mid i \in \inv \} = \{M'(e_l)^{-1} \mid i \in \inv \} \subseteq \{M'(e_l)^{-1} \mid i \in \inv'\}.\] Thus, the image of $\CC(\Sigma)$ under $f$ lies in $\CC(\Sigma')$ 
	and by Lemma \ref{L:full subseeds}, $f$ respects the gradings of all cluster variables and thus it is a graded $\mathbb{K}$-algebra homomorphism. 
	Axioms \ref{CM1} and \ref{CM2prime} are satisfied by definition and axiom \ref{CM3} is satisfied by Lemma \ref{L:full subseeds}.
\end{proof}

For any given rooted cluster algebra $\CC(\Sigma)$ we can build a linear system $\{\CC(\Sigma_i)\}_{i \in \ZZ}$ of rooted cluster algebras whose initial seeds are finite full subseeds $\Sigma_i$ of $\Sigma$ such that for all $i \in \ZZ$, the seeds $\Sigma_i$ and $\Sigma$ are only connected by coefficients of $\Sigma_i$. Further, we can construct it in a way such that for all $i \leq j$ the seed $\Sigma_i$ is a full subseed of $\Sigma_j$ and the two are connected only by coefficients of $\Sigma_i$. This construction yields a linear system of rooted cluster algebras of finite rank which has the desired rooted cluster algebra $\CC(\Sigma)$ as its colimit.

\begin{theorem}\label{T:connected colimit}
Every rooted graded quantum cluster algebra is isomorphic to a colimit of rooted graded quantum cluster algebras of finite rank in the category $\Clus$.
\end{theorem}

\begin{proof}
Let $\CC(\Sigma)$ be a rooted cluster algebra with initial seed $\Sigma = (M, \tilde{B}, G, \var, \ex, \inv)$. Since we can write any cluster algebra as a coproduct of connected cluster algebras, without loss of generality we can assume that $\Sigma$ is connected.

We construct a linear system of rooted cluster algebras as follows. Pick $i_0 \in \var\setminus \ex$ and set
\[
	\Sigma_0 = (M_0,[0],G_0, \var_0, \ex_0, \inv_0)
\]
with 
	\begin{itemize}
		\item $\var_0 = \{i_0\}$
		\item $\ex_0 = \emptyset$
		\item $\inv_0 = \begin{cases} \{ i_{0} \} & \text{if}\ i_{0}\in \inv \\ \emptyset & \text{otherwise} \end{cases}$
		\item $M_0 \colon \mathbb{Z}^{\var_0} \to \curly{F}_0$, where $\curly{F}_0$ is the subdivision algebra of $\curly{F}$ generated by $M(e_{i_0})$ and $M_0(1) = M(e_{i_0})$
		\item $G_0 = G_{i_0}$. 
	\end{itemize}
Inductively define full subseeds $\Sigma_i$ of $\Sigma$ by setting
\[
	\Sigma_{i+1} = (M_{i+1},\tilde{B}^{i+1},G_{i+1}, \var_{i+1}, \ex_{i+1}, \inv_{i+1}),
\]
where
	\begin{itemize}
		\item $\var_{i+1} = \var_i \cup \{l \in \var \mid \tilde{B}_{kl} \neq 0 \text{ for some } k \in \var_i\}$
		\item $\ex_{i+1} = \var_i \cap \ex$
		\item $\inv_{i+1} = \var_i \cap \inv$
		\item $M_{i+1} \colon \mathbb{Z}^{\var_{i+1}} \to \curly{F}_{i+1}$, where $\curly{F}_{i+1}$ is the subdivision algebra of $\curly{F}$ generated by $\{M(e_{k})\mid k \in \var_{i+1}\}$ and 
			$M_{i+1}(a) = M(a)$ for all $a \in \mathbb{Z}^{\var_{i+1}}$, where we can view $a$ as a vector in $\mathbb{Z}^{\var}$ by filling up the vector with zeroes.
		\item $\tilde{B}^{i+1} = \tilde{B} \mid_{\var_{i+1}}^{\var_{i+1}}$
		\item $(G_{i+1}) = G \mid_{\var_{i+1}}$. 
	\end{itemize}
Note that because $\tilde{B}$ is sign-skew-symmetric, $\tilde{B}_{kl} \neq 0$ is equivalent to $\tilde{B}_{lk} \neq 0$. 

It is routine to check that $\Sigma_{i}$ is a graded quantum seed.  The required properties are inherited from those of $\Sigma$, noting that we implicitly define the chain of subalgebras $\curly{F}_{0} \subseteq \curly{F}_{1} \subseteq \dotsm$ such that the toric frame conditions hold.

The key point is to observe that since these seeds are full subseeds, for $k\in \ex_{i}$ we have $\supp(\tilde{B}^{i}_{k})=\supp(\tilde{B}_{k})$, that is, at any point where we need to verify a compatibility condition with the exchange matrix, the restricted exchange matrix differs from the full one only by the removal of zeroes outside the relevant indexing set.

Because $\tilde{B}$ is locally finite, for all $i \geq 0$ the indexing set $\var_i$ in the seed $\Sigma_i$ is finite.
By definition, the seed $\Sigma_i$ is a full subseed of the seed $\Sigma_{i+1}$ for all $i \geq 0$ and all the seeds $\Sigma_i$ are full subseeds of $\Sigma$. 

\vfill
\pagebreak
We now want to show that for all $0 \leq i < j$ the seeds $\Sigma_i$ and $\Sigma_j$ are connected only by coefficients of $\Sigma_i$. Let $k \in \ex_i$ with $\tilde{B}^j_{kl} \neq 0$. Since $k \in \ex_i \subseteq \var_{i-1}$ and $\tilde{B}^j_{kl} = \tilde{B}_{kl}$ it follows that $l \in \{t \in \var \mid \tilde{B}_{tm} \neq 0 \text{ for some } m \in \var_{i-1}\} \subseteq \var_i$ and thus the seeds $\Sigma_i$ and $\Sigma_j$ are only connected by coefficients in $\Sigma_i$. Analogously one shows that the seeds $\Sigma_i$ and $\Sigma$ are only connected by coefficients in $\Sigma_i$.

By Theorem \ref{T:inclusion morphism} for $0 \leq i \leq j$, the natural inclusion $f_{ij} \colon \var_i \to \var_j$ gives rise to a rooted cluster morphism $f_{ij} \colon \CC(\Sigma_i) \to \CC(\Sigma_j)$. For all $0 \leq i \leq j \leq k$ we have $f_{jk} \circ f_{ij} = f_{ik}$ and $f_{ii} = \id_{\CC(\Sigma_i)}$, so the morphisms form a linear system of rooted cluster algebras of finite rank. Further, again by Theorem \ref{T:inclusion morphism}, for $i \geq 0$ the natural inclusion $f_i \colon\var_i \to \var$ gives rise to a rooted cluster morphism $f_i \colon\CC(\Sigma_i) \to \CC(\Sigma)$. We show that $\CC(\Sigma)$ together with the maps $f_i \colon \CC(\Sigma_i) \to \CC(\Sigma)$ for $i \geq 0$ is in fact the colimit of this linear system in the category of rooted cluster algebras. 

Because we assumed $\Sigma$ to be connected and the toric frames agree, we have $\cl(\Sigma) = \bigcup_{i \geq 0} \cl(\Sigma_i)$. Since every exchange relation in $\CC(\Sigma)$ is an exchange relation in $\CC(\Sigma_i)$ for all $i$ large enough (by virtue of the exchange matrices $\tilde{B}^i$ being arbitrarily large restrictions of the exchange matrix $\tilde{B}$ and the toric frames agreeing), any fixed element of $\CC(\Sigma)$ is contained in $\CC(\Sigma_i)$ for all $i$ sufficiently large.

Let $\Sigma'= (M', \tilde{B}', G', \var', \ex', \inv')$ be a seed such that for all $i \geq 0$ there are rooted cluster morphisms $g_i \colon \CC(\Sigma_i) \to \CC(\Sigma')$ compatible with the linear system $f_{ij} \colon \CC(\Sigma_i) \to \CC(\Sigma_j)$. 
We define a $\mathbb{K}$-algebra homomorphism $f \colon \CC(\Sigma) \to \CC({\Sigma'})$ by $f(x) = g_i(x)$, whenever $x \in \CC(\Sigma_i)$, i.e.\ it is the unique $\mathbb{K}$-algebra homomorphism making the following diagram commute.

\[
	\xymatrix{
	&\CC({\Sigma'})& \\
	& \CC(\Sigma) \ar@{-->}[u]^f& \\
	\CC(\Sigma_i) \ar[ruu]^{g_i} \ar[ru]_{f_i} \ar[rr]^{f_{ij}} && \CC(\Sigma_j) \ar[luu]_{g_j} \ar[lu]^{f_j}
}
\]
Since the morphisms $g_i$ all respect the gradings of $\CC(\Sigma_i)$, and these are compatible with the grading of $\CC(\Sigma)$, it follows that $f$ is a graded $\mathbb{K}$-algebra homomorphism.

For every $x \in \cl(\Sigma)$, there exists a $k \geq 0$ such that $x \in \cl(\Sigma_i)$. Thus $f(x) = g_i(x)$ for all $i \geq k$ lies in $\cl(\Sigma')$, because $g_i$ is a rooted cluster morphism for all $i \geq 0$. Thus the ring homomorphism $f$ satisfies axiom \ref{CM1}, and analogously axiom \ref{CM2prime}. Let now $(x_1, \ldots, x_l)$ be a $(f,\Sigma,{\Sigma'})$-biadmissible sequence and let $y \in \cl(\Sigma)$ such that $f(y) \in \cl(\Sigma')$. Then there exists an $i \geq 0$ such that $y \in \cl(\Sigma_i)$ and the sequence $(x_1, \ldots, x_l)$ is $(g_i,\Sigma_i,\Sigma')$-biadmissible. Thus we get 
\begin{align*}
	f(\mu_{x_l} \circ \ldots \circ \mu_{x_1}(y)) &= f \circ f_i(\mu_{x_l} \circ \ldots \circ \mu_{x_1}(y))\\
	&= g_i(\mu_{x_l} \circ \ldots \circ \mu_{x_1}(M(e_i)))= \mu_{g_i(x_l)} \circ \ldots \circ \mu_{g_i(x_1)}(g_i(y)))\\
	&= \mu_{f(x_l)} \circ \ldots \circ \mu_{f(x_1)}(f(y)).
\end{align*}
Therefore the ring homomorphism $f$ satisfies \ref{CM3} and is a rooted cluster morphism. Thus $\CC(\Sigma)$ satisfies the required universal property.
\end{proof}

\section{Infinite versions of homogeneous coordinate rings of Grassmannians via the Pl\"ucker embedding}

Let $[a,b]=\{ i \mid a\leq i \leq b \}$ with the convention that $[a,b]=\emptyset$ if $b<a$.  Fix $k$ and define 
\[ I_{ij} = [1,k-i] \sqcup [(k+j)-(i-1),k+j]. \]
Then the quantized coordinate ring over the Grassmannian $\mathcal{O}_{q}(\Gr(k,n))$ is a graded quantum cluster algebra over $\mathbb{K}$ with initial cluster
\[ \Sigma_{\Gr}(k,n)\stackrel{\mathrm{def}}{=} \{ \Delta_{q}^{I_{ij}} \mid (i,j)\in ([1,k]\times [1,n-k])\cup (0,0) \}. \]
Here $\Delta_{q}^{I_{ij}}$ denotes the quantum Pl\"{u}cker coordinate with indexing set $I_{ij}$; see \cite{GradedQCAs-Grkn} for detailed definitions.  

We use the indexing set $\var=([1,k]\times [1,n-k])\cup (0,0)$ for the cluster.  The mutable indices are then the subset $\ex=[1,k-1]\times [1,n-k-1]$.  We set $\inv=\emptyset$.  

Notice that for all $(i,j)\in \ex$, we have that $I_{ij} \cap \{n\}=\emptyset$; that is, no mutable Pl\"{u}cker coordinate in this cluster has an indexing set involving $n$.  This property will be key in what follows.

The exchange matrix for this cluster is the matrix associated to the quiver for which there is an arrow $(i_{1},j_{1})\to (i_{2},j_{2})$ if and only if the indices satisfy the following conditions:
\begin{enumerate}
\item $(i_{1},j_{1})=(0,0)$, $(i_{2},j_{2})=(1,1)$
\item $(i_{2},j_{2})=(i_{1}+1,j_{1})$
\item $(i_{2},j_{2})=(i_{1},j_{1}+1)$
\item $(i_{2},j_{2})=(i_{1}-1,j_{1}-1)$
\end{enumerate}
Here, we simply ignore any arrows whose tail index does not belong to $\mathbf{var}$ but we do allow arrows between frozen vertices.  These have no effect on the cluster algebra but are necessary for what follows.

This cluster is illustrated for $(k,n)=(3,7)$ and $(k,n)=(3,8)$ in Figure~\ref{fig:initialseed}.

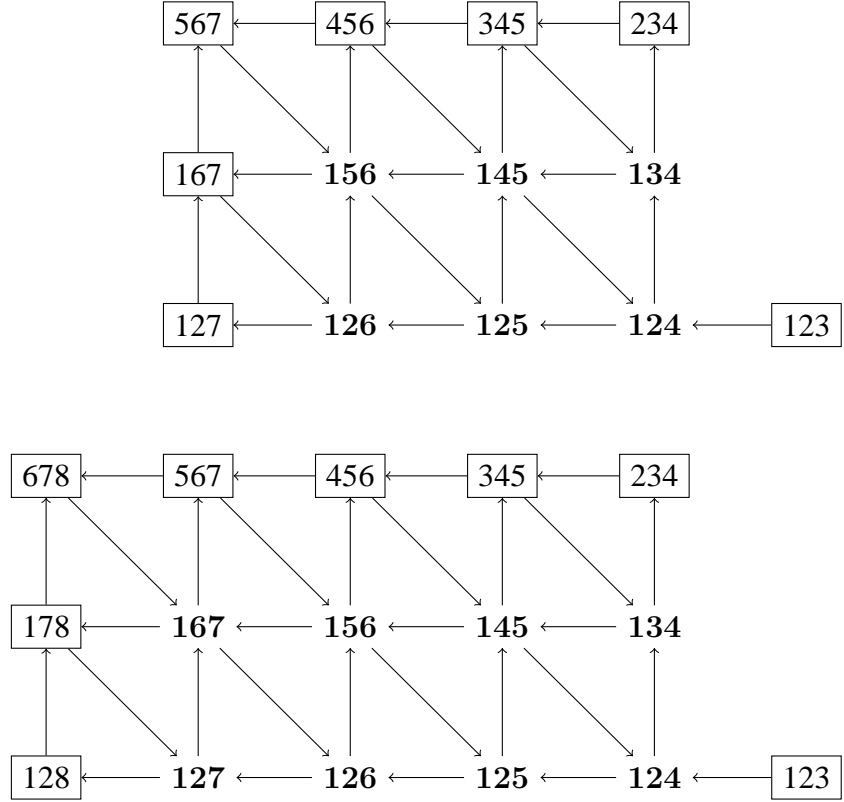
\begin{figure}
\begin{center}
{
\scalebox{1}{
\begin{tikzpicture}[on grid,node distance=2cm]

\begin{scope}[xshift=2cm]
\node (156) at (0,0) {$\mathbf{156}$};
\node (126) [below=of 156] {$\mathbf{126}$};
\node (145) [right=of 156] {$\mathbf{145}$};
\node (125) [below=of 145] {$\mathbf{125}$};
\node (134) [right=of 145] {$\mathbf{134}$};
\node (124) [below=of 134] {$\mathbf{124}$};
\node (123) [right=of 124,rectangle,draw=black]  {123};
\node (127) [left=of 126,rectangle,draw=black] {127};
\node (167) [left=of 156,rectangle,draw=black] {167};
\node (567) [above=of 167,rectangle,draw=black] {567};
\node (456) [above=of 156,rectangle,draw=black] {456};
\node (345) [above=of 145,rectangle,draw=black] {345};
\node (234) [above=of 134,rectangle,draw=black] {234};

\draw[->] (124) to (125);
\draw[->] (125) to (126);
\draw[->] (126) to (127);

\draw[->] (134) to (145);
\draw[->] (145) to (156);
\draw[->] (156) to (167);

\draw[->] (124) to (134);
\draw[->] (134) to (234);

\draw[->] (125) to (145);
\draw[->] (145) to (345);

\draw[->] (126) to (156);
\draw[->] (156) to (456);

\draw[->] (167) to (126);
\draw[->] (567) to (156);

\draw[->] (156) to (125);
\draw[->] (456) to (145);

\draw[->] (145) to (124);
\draw[->] (345) to (134);

\draw[->] (123) to (124);

\draw[->] (234) to (345);
\draw[->] (345) to (456);
\draw[->] (456) to (567);

\draw[->] (127) to (167);
\draw[->] (167) to (567);
\end{scope}

\begin{scope}[yshift=-6cm]
\node (167) at (0,0) {$\mathbf{167}$};
\node (127) [below=of 167] {$\mathbf{127}$};
\node (156) [right=of 167] {$\mathbf{156}$};
\node (126) [below=of 156] {$\mathbf{126}$};
\node (145) [right=of 156] {$\mathbf{145}$};
\node (125) [below=of 145] {$\mathbf{125}$};
\node (134) [right=of 145] {$\mathbf{134}$};
\node (124) [below=of 134] {$\mathbf{124}$};
\node (123) [right=of 124,rectangle,draw=black]  {123};
\node (128) [left=of 127,rectangle,draw=black] {128};
\node (178) [left=of 167,rectangle,draw=black] {178};
\node (678) [above=of 178,rectangle,draw=black] {678};
\node (567) [above=of 167,rectangle,draw=black] {567};
\node (456) [above=of 156,rectangle,draw=black] {456};
\node (345) [above=of 145,rectangle,draw=black] {345};
\node (234) [above=of 134,rectangle,draw=black] {234};

\draw[->] (123) to (124);
\draw[->] (124) to (125);
\draw[->] (125) to (126);
\draw[->] (126) to (127);
\draw[->] (127) to (128);

\draw[->] (134) to (145);
\draw[->] (145) to (156);
\draw[->] (156) to (167);
\draw[->] (167) to (178);

\draw[->] (145) to (124);
\draw[->] (156) to (125);
\draw[->] (167) to (126);

\draw[->] (124) to (134);
\draw[->] (125) to (145);
\draw[->] (126) to (156);
\draw[->] (127) to (167);

\draw[->] (134) to (234);
\draw[->] (145) to (345);
\draw[->] (156) to (456);
\draw[->] (167) to (567);

\draw[->] (178) to (127);
\draw[->] (678) to (167);
\draw[->] (567) to (156);
\draw[->] (456) to (145);
\draw[->] (345) to (134);

\draw[->] (234) to (345);
\draw[->] (345) to (456);
\draw[->] (456) to (567);
\draw[->] (567) to (678);

\draw[->] (128) to (178);
\draw[->] (178) to (678);

\end{scope}

\end{tikzpicture}
}
}
\end{center}
\caption{Initial clusters for quantum cluster algebra structures on $\mathcal{O}_{q}(\Gr(3,7))$ (upper) and $\mathcal{O}_{q}(\Gr(3,8))$ (lower).\label{fig:initialseed}}
\end{figure}

The grading we choose is the $\mathbb{Z}$-grading $G=\mathbf{1}$, i.e.\ $G_{i}=1$ for all $1\leq i\leq |\mathbf{var}|$.  

We will not need to explicitly describe the associated toric frame: it is fully determined by the quasi-commutation rules for quantum Pl\"{u}cker coordinates described by Scott (\cite{Scott-QMinors}).

Then we have the following proposition.

\begin{proposition} Let $\iota_{n}\colon \mathcal{O}_{q}(\Gr(k,n))\to \mathcal{O}_{q}(\Gr(k,n+1))$, $\iota_{n}(\Delta_{q}^{I})=\Delta_{q}^{I}$, be the natural embedding.  Then $\iota_{n}\colon (\mathcal{O}_{q}(\Gr(k,n)),\Sigma_{\Gr}(k,n)) \to (\mathcal{O}_{q}(\Gr(k,n+1)),\Sigma_{\Gr}(k,n+1))$ is a rooted (graded quantum) cluster morphism.
\end{proposition}

\begin{proof} It is straightforward to check that $\Sigma_{\Gr}(k,n)$ is a full subseed of $\Sigma_{\Gr}(k,n+1)$, connected only by coefficients.  Then the result follows immediately from Theorem~\ref{T:inclusion morphism}.
\end{proof}

The map $\iota_{7}\colon \mathcal{O}_{q}(\Gr(3,7))\to \mathcal{O}_{q}(\Gr(3,8))$ can be visualised via Figure~\ref{fig:initialseed}.  In particular no exchangeable variable from $\mathcal{O}_{q}(\Gr(3,7))$, considered in $\mathcal{O}_{q}(\Gr(3,8))$ is connected to the ``new'' variables indexed by $128$, $178$ and $678$, this being the ``connected only by coefficients'' condition.

Then the main colimit theorem above yields the following.

\begin{theorem} The family of rooted cluster morphisms $\{\iota_{n} \mid n\geq 1\}$ has a colimit, a rooted graded quantum cluster algebra which we denote by $\mathcal{O}_{q}(\Gr(k,\infty))$.  

This algebra is generated by the set $\{ \Delta_{q}^{I} \mid I\subset \mathbb{N}, |I|=k \}$ and has as relations all (higher) quantum Pl\"{u}cker relations.
\end{theorem}

\begin{proof} It is clear that $\mathcal{O}_{q}(\Gr(k,\infty))$ with the stated properties may be constructed as a colimit of algebras arising from the system $\{ \iota_{n} \}$.  Then Theorem~\ref{T:connected colimit} shows us that this colimit has precisely the colimit cluster structure, i.e.\ that the colimit is compatible with the cluster structures on each finite Grassmannian.
\end{proof}

This generalises to arbitrary $k$ the construction of the authors in \cite{Grabowski-Gratz} for $k=2$ (see also \cite[\S 4.4]{Gratz-Colimits}).\vspace{-0.9em}

\small

\bibliographystyle{amsplain}
\bibliography{biblio}\label{references}

\providecommand{\bysame}{\leavevmode\hbox to3em{\hrulefill}\thinspace}
\providecommand{\MR}{\relax\ifhmode\unskip\space\fi MR }
\providecommand{\MRhref}[2]{%
  \href{http://www.ams.org/mathscinet-getitem?mr=#1}{#2}
}
\providecommand{\href}[2]{#2}
\begin{thebibliography}{10}

\bibitem{ADS}
Ibrahim Assem, Gr{\'e}goire Dupont, and Ralf Schiffler, \emph{On a category of
  cluster algebras}, J. Pure Appl. Algebra \textbf{218} (2014), no.~3,
  553--582.

\bibitem{BFZ-CA3}
Arkady Berenstein, Sergey Fomin, and Andrei Zelevinsky, \emph{Cluster algebras.
  {III}. {U}pper bounds and double {B}ruhat cells}, Duke Math. J. \textbf{126}
  (2005), no.~1, 1--52.

\bibitem{BZ-QCA}
Arkady Berenstein and Andrei Zelevinsky, \emph{Quantum cluster algebras}, Adv.
  Math. \textbf{195} (2005), no.~2, 405--455.

\bibitem{FZ-CA1}
Sergey Fomin and Andrei Zelevinsky, \emph{Cluster algebras. {I}.
  {F}oundations}, J. Amer. Math. Soc. \textbf{15} (2002), no.~2, 497--529
  (electronic).

\bibitem{GoodearlYakimovQCA}
Ken~R. Goodearl and Milen Yakimov, \emph{Quantum cluster algebra structures on
  quantum nilpotent algebras}, Memoirs Amer. Math. Soc. (2013), to appear.

\bibitem{Gorsky}
Mikhail Gorsky, \emph{On {Y}oung diagrams, flips and cluster algebras of type
  {$A$}}, preprint posted at arXiv:1106.2458.

\bibitem{GradedClusterAlgebras}
Jan~E. Grabowski, \emph{Graded cluster algebras}, J. Algebraic Combin. (2014),
  to appear.

\bibitem{Grabowski-Gratz}
Jan~E. Grabowski and Sira Gratz, \emph{Cluster algebras of infinite rank}, J.
  Lond. Math. Soc. (2) \textbf{89} (2014), no.~2, 337--363, With an appendix by
  Michael Groechenig.

\bibitem{GradedQCAs-Grkn}
Jan~E. Grabowski and St{\'e}phane Launois, \emph{Graded quantum cluster
  algebras and an application to quantum {G}rassmannians}, Proc. Lond. Math.
  Soc. (3) \textbf{109} (2014), no.~3, 697--732.

\bibitem{GrabowskiLaunois-Lifting}
Jan~E. Grabowski and St\'{e}phane Launois, \emph{Lifting quantum cluster
  algebra structures}, in preparation, 2015.

\bibitem{Gratz-Colimits}
Sira Gratz, \emph{Cluster algebras of infinite rank as colimits}, Math. Z.
  (2014), to appear.

\bibitem{HL}
David Hernandez and Bernard Leclerc, \emph{Monoidal categorifications of
  cluster algebras of type {$A$} and {$D$}}, Symmetries, integrable systems and
  representations, Springer Proc. Math. Stat., vol.~40, Springer, Heidelberg,
  2013, pp.~175--193.

\bibitem{HJ}
Thorsten Holm and Peter J{\o}rgensen, \emph{On a cluster category of infinite
  {D}ynkin type, and the relation to triangulations of the infinity-gon}, Math.
  Z. \textbf{270} (2012), no.~1-2, 277--295.

\bibitem{IT:cyclic}
Kiyoshi Igusa and Gordana Todorov, \emph{Cluster categories coming from cyclic
  posets}, Comm. Algebra \textbf{43} (2015), no.~10, 4367--4402.

\bibitem{IT:continuous}
\bysame, \emph{Continuous cluster categories {I}}, Algebr. Represent. Theory
  \textbf{18} (2015), no.~1, 65--101.

\bibitem{Ndoune}
Ndoun\'{e} Ndoun\'{e}, \emph{Cluster algebras arising from the infinity-gon},
  preprint posted at arXiv:1306.6120.

\bibitem{Scott-QMinors}
Josh Scott, \emph{Quasi-commuting families of quantum minors}, J. Algebra
  \textbf{290} (2005), no.~1, 204--220.

\end{thebibliography}

\normalsize

\end{document}